\documentclass[11pt,reqno,draft]{amsart}
\usepackage{graphicx}
\usepackage{verbatim}
\usepackage{textcomp}
\usepackage{amssymb}
\usepackage{cite}
\usepackage{amsmath}
\usepackage{latexsym}
\usepackage{amscd}
\usepackage{amsthm}
\usepackage{mathrsfs}
\usepackage{xypic}
\usepackage{bm}
\usepackage{url}
\usepackage{hyperref}

\newtheorem{thm}{Theorem}[section]
\newtheorem{corr}[thm]{Corollary}

\theoremstyle{definition}

\theoremstyle{remark}
\newtheorem{rem}{Remark}[section]
\numberwithin{equation}{section}
\begin{document}
\title[Estimates for eigenvalues of $\mathcal{L}_r$
operator] {Estimates for eigenvalues of $\mathcal{L}_r$ operator on
self-shrinkers }
\author{Guangyue Huang}
\address{College of Mathematics and Information Science, Henan Normal
University, Xinxiang, Henan 453007, People's Republic of China}
\email{hgy@henannu.edu.cn }
\author{Xuerong Qi}
\address{School of Mathematics and Statistics, Zhengzhou University, Zhengzhou, Henan 450001,
People's Republic of China} \email{xrqi@zzu.edu.cn}
\author{Hongjuan Li}
\address{College of Mathematics and Information Science, Henan Normal
University, Xinxiang, Henan 453007, People's Republic of China}
\email{hjli10@126.com} 

\maketitle

\begin{abstract}
Let $x: M\rightarrow \mathbb{R}^{N}$ be an $n$-dimensional compact
self-shrinker in $\mathbb{R}^N$ with smooth boundary $\partial\Omega$.
In this paper, we study eigenvalues of the operator
$\mathcal{L}_r$ on $M$, where $\mathcal{L}_r$ is defined by
$$\mathcal{L}_r=e^{\frac{|x|^2}{2}}{\rm
div}(e^{-\frac{|x|^2}{2}}T^r\nabla\cdot)$$ with $T^r$ denoting a
positive definite (0,2)-tensor field on $M$. We obtain ``universal" inequalities for eigenvalues of the operator $\mathcal{L}_r$. These inequalities generalize
the result of Cheng and Peng in \cite{ChengPeng2013}.
 Furthermore, we also consider the case that
equalities occur.
\end{abstract}

\section{Introduction}

A self-shrinker is an immersion $x: M\rightarrow \mathbb{R}^{N}$ of
a smooth $n$-dimensional manifold $M$ into the Euclidean space
$\mathbb{R}^{N}$ which satisfies the quasilinear elliptic system:
\begin{equation}\label{1Introduction-Section1}
n\mathbf{H}=-x^{\perp},
\end{equation} where $\mathbf{H}$ denotes the mean curvature vector field
of the immersion and $\perp$ is the projection onto the normal
bundle of $M$. Self-shrinkers play an important role in the study of
the mean curvature flow since they not only correspond to solutions
of the mean curvature flow, but also describe all possible blow ups
at a given singularity of the mean curvature flow. For more
information on self-shrinkers and singularities of mean curvature
flow, we refer the readers to
\cite{White2002,Huisken1993,Huisken1990,Colding2012} and references
therein.

In \cite{Colding2012}, Colding and Minicozzi introduced the
following differential operator $\mathcal{L}$ and used it to study
self-shrinkers:
\begin{equation}\label{1Introduction-Section2}
\mathcal{L}=\Delta-\langle x,\nabla\cdot \rangle,
\end{equation}
where $\Delta$, $\nabla$ denote the Laplacian, the  gradient operator on
the self-shrinker, respectively, $\langle,\rangle$ stands for the
standard inner product in $\mathbb{R}^{N}$. It is easy to see that
the operator can be written as
\begin{equation}\label{1Introduction-Section3}
\mathcal{L}=e^{\frac{|x|^2}{2}}{\rm
div}(e^{-\frac{|x|^2}{2}}\nabla\cdot),
\end{equation} where ${\rm
div}$ is the divergent operator on the self-shrinker. Obviously, for a compact self-shrinker $M^{n}$, the
operator $\mathcal{L}$ is self-adjoint with respect to the measure $e^{-\frac{|x|^{2}}{2}}dv$.
That is,
\begin{equation}\label{1Introduction-Section4}
\int\limits_Mu\mathcal{L}v\,e^{-\frac{|x|^2}{2}}dv=\int\limits_Mv\mathcal{L}u\,e^{-\frac{|x|^2}{2}}dv
=-\int\limits_M\langle\nabla u,\nabla
v\rangle\,e^{-\frac{|x|^2}{2}}dv
\end{equation} holds for any $u|_{\partial
M}=v|_{\partial M}=0$. Let $T$ be a positive definite (0,2)-tensor field on $M$ and $f\in C^1(M)$. The following elliptic operator in
divergence form
\begin{equation}\label{1Introduction-Section5}
\mathcal{L}^{(f,T)}=e^{f}{\rm div}(e^{-f}T\nabla\cdot)
\end{equation} is very interesting. For any two smooth functions $u,v$ on $M$ with $u|_{\partial
M}=v|_{\partial M}=0$, by the stokes formula, we have
\begin{equation}\label{1Introduction-Section12}
\int\limits_Mu\mathcal{L}^{(f,T)}v\,d\mu=\int\limits_Mv\mathcal{L}^{(f,T)}u\,d\mu
=-\int\limits_M\langle\nabla u,T\nabla v\rangle\,d\mu,
\end{equation} where $d\mu=e^{-f}dv$.
That is to say, the operator $\mathcal{L}^{(f,T)}$  is self-adjoint
on the space of smooth functions on $M$ vanishing on $\partial M$
with respect to the $L^2$ inner product under the measure
$d\mu=e^{-f}dv$. Therefore, the eigenvalue problem
\begin{equation}\label{1Introduction-Section13}
\left\{\begin{array}{l} \mathcal{L}^{(f,T)}(u)=-\lambda u, \ {\rm in}\  M, \\
u|_{\partial M}=0,
\end{array}\right.
\end{equation} has a real and discrete spectrum:
$$0<\lambda_1\leq\lambda_2\leq \lambda_3\leq \cdots\rightarrow+\infty.$$
We put $\lambda_0=0$ if $\partial M=\emptyset$. Here each eigenvalue is
repeated according to its multiplicity.

In particular, when $T$ is an identity map $I$, the operator
$\mathcal{L}^{(f,T)}$ becomes the the drifting Laplacian
$\Delta_f=\Delta-\langle\nabla f,\nabla\rangle$, many interesting
estimates have been obtained (cf.
\cite{XiaXu2014}); when $T$ is an identity map and $f=0$, then
$\mathcal{L}^{(f,T)}$ becomes the Laplace operator. Let $A$ be the
shape operator of an immersion $x: M\rightarrow \mathbb{R}^{n+1}(c)$
of an $n$-dimensional hypersurface $M$ into an $(n+1)$-dimensional
space form $\mathbb{R}^{n+1}(c)$ of constant sectional curvature
$c$. Recall that using the characteristic polynomial of $A$, we can
define the elementary symmetric function $S_r$ as follows
\begin{equation}\label{1Introduction-Section6}
{\rm det}(tI-A)=\sum_{r=0}^n(-1)^rS_rt^{n-r},
\end{equation} where $r$th mean curvature $H_r$ is defined by
$S_r=\binom{n}{r}H_r$. Then $H_1$ is the mean curvature $H$ and
$H_0=1$. The classical Newton transformation $T^r$ are inductively
defined by
\begin{equation}\label{1Introduction-Section7}\aligned
T^0=&I,\\
T^r=&S_rI-T^{r-1}A.
\endaligned\end{equation}
For each $T^r$ defined by \eqref{1Introduction-Section7}, we have a
second order differential operator $L_r$ defined by
\begin{equation}\label{1Introduction-Section8}
L_r={\rm div}(T^r\nabla\cdot).
\end{equation} Clearly, the $L_r$ operator can be seen as a special
case of $\mathcal{L}^{(f,T)}$ by substituting $T$ and $f$ by $T^r$
and $0$, respectively. In particular, $L_0=\Delta$, $L_1$ becomes
the operator $\Box$ introduced by Cheng-Yau in \cite{Cheng77}.
Estimates on eigenvalues of $L_r$ operator were studied by many
mathematicians. For example, in \cite{Alencar1993}, Alencar, do
Carmo and Rosenberg derived the upper bound of the first eigenvalue
of the operator $L_r$ on compact hypersurfaces of the Euclidean space
$\mathbb{R}^{n+1}$:
\begin{equation}\label{1Introduction-Section9}
\lambda_1\int\limits_M H_r\,dv\leq c(r)\int\limits_M
H_{r+1}^2\,dv
\end{equation} and equality holds if and only if $M$ is a sphere,
where $c(r)=(n-r)\binom{n}{r}$. For the first eigenvalue of $L_r$ on
hypersurfaces of space forms, see
\cite{Alencar2001,Alias2004,LiWang2012,Cheng2008,Grosjean2000,HuangQi2015} and
references therein.

For submanifolds of $\mathbb{R}^N$,  we let $\{e_A\}_{A=1}^N$
be an orthonormal basis along $M$ such that $\{e_i\}_{i=1}^n$ are
tangent to $M$ and $\{e_\alpha\}_{\alpha=n+1}^N$ are normal to $M$. Then
$A_{ij}=\sum_{\alpha=n+1}^Nh_{ij}^\alpha e_\alpha$. If
$r\in\{0,1,\cdots,n-1\}$ is even, for any smooth function $u$, the
operator $L_r$ is defined by (see \cite{Cao2007})
\begin{equation*}\label{0Section4}
L_r(u)={\rm div}(T^r\nabla u)=\sum_{i,j}T^r_{ij}u_{ij}
\end{equation*}  since $T^r$ is symmetric and divergence free.
Here $T^r$ is given by
\begin{equation}\label{0Section6}
T^r_{ij}=\frac{1}{r!}\sum_{\mbox{\tiny$\begin{array}{c}
i_1\cdots i_{r} \\
j_1\cdots j_{r}\end{array}$}}\delta_{i_1\cdots i_r i}^{j_1\cdots
j_r j}\langle A_{i_1j_1},A_{i_2j_2}\rangle\cdots\langle
A_{i_{r-1}j_{r-1}},A_{i_rj_r}\rangle;
\end{equation}
$\delta_{i_1\cdots i_r i}^{j_1\cdots j_r j}$ is the generalized
Kronecker symbol. By substituting $T$ and $f$ by $T^r$ and $\frac{|x|^2}{2}$,
respectively, the operator $\mathcal{L}^{(f,T)}$ becomes
\begin{equation}\label{1Introduction-Section11}
\mathcal{L}_r:=\mathcal{L}^{(\frac{|x|^2}{2},T^r)}=L_r-\langle
x,T^{r}\nabla\cdot\rangle,
\end{equation} which is important in the study of
self-shrinkers. For example, when $r=0$, $\mathcal{L}_0$ becomes the
$\mathcal{L}$ operator on the self-shrinker given by
\eqref{1Introduction-Section2}.

In this paper, we assume that $T^{r}$ is positive definite on $M$, for some even integer $r\in\{0,1,\cdots,n-1\}$, namely
the operator $\mathcal{L}_r$ is elliptic. Denote by $\lambda_{i}$ the $i$-th eigenvalue of the following eigenvalue problem:
\begin{equation}\label{maith-1}
\left\{\begin{array}{l} \mathcal{L}_r(u)=-\lambda u,  \ {\rm in}\ M; \\
u|_{\partial M}=0,
\end{array}\right.
\end{equation}
and let $u_i$ be the normalized
eigenfunction corresponding to $\lambda_{i}$ such that
$\{u_i\}_1^\infty$ becomes an orthonormal basis of $L^2(M)$ under
the weighted measure $d\mu=e^{-\frac{|x|^{2}}{2}}dv$, that is
\begin{equation}\label{2Section-Ineq1}
\left\{\begin{array}{l} \mathcal{L}_{r}(u_i)=-\lambda_{i} u_i, \ {\rm in}\ M;\\
u_i|_{\partial M}=0;\\
 \int_{M}u_iu_j\,d\mu=\delta_{ij}.
\end{array}\right.
\end{equation}
 The purpose of this paper is to study eigenvalues of the operator
$\mathcal{L}_r$ on a compact self-shrinker $M$ of
$\mathbb{R}^{N}$. We
proved the following

\begin{thm}\label{mainth1}
Let $x: M\rightarrow \mathbb{R}^{N}$ be an $n$-dimensional compact
self-shrinker in $\mathbb{R}^N$ with smooth boundary $\partial M$. Assume that $T^{r}$ is positive definite on $M$, for some even integer $r\in\{0,1,\cdots,n-1\}$. Denote
by $\xi$ a positive lower bound of $T^{r}$,
then the eigenvalues $\lambda_i$ of the eigenvalue problem \eqref{maith-1} satisfy
\begin{equation}\label{maith-2}\aligned
\sum_{i=1}^k(\lambda_{k+1}-\lambda_i)^2
\leq\frac{4(n-r)}{n^2}\max_M(S_r)\sum_{i=1}^k
(\lambda_{k+1}-\lambda_i)\left(\frac{\lambda_i}{\xi}+
\frac{2n-\min\limits_{M}|x|^2}{4}\right)
\endaligned\end{equation}
and
\begin{equation}\label{maith-3}
\sum\limits_{i=1}^n\sqrt{\lambda_{i+1}-\lambda_{1}}
\leq2\sqrt{(n-r)\max_M(S_r)\left(\frac{\lambda_1}{\xi}+
\frac{2n-\min\limits_{M}|x|^2}{4}\right)},
\end{equation}
where $\xi$ denotes $S_r$ denotes the $r$th mean curvature function of $x$.

\end{thm}

\begin{rem}
In particular, we have $S_0=1$. Hence, we obtain Theorem 1.1 of
Cheng and Peng in \cite{ChengPeng2013} by taking $r=0$ in
\eqref{maith-2}.
\end{rem}

When $\partial M=\emptyset$, the closed eigenvalue
problem
\begin{equation}\label{maith-2}
\mathcal{L}_r(u)=-\lambda u
\end{equation}
has the following discrete spectrum:
$$0=\lambda_0<\lambda_1
\leq\lambda_2\leq \lambda_3\leq
\cdots\rightarrow+\infty.$$ We have the following
results:

\begin{thm}\label{mainth2}
Let $x: M\rightarrow \mathbb{R}^{N}$ be an $n$-dimensional compact
self-shrinker in $\mathbb{R}^N$ without boundary. Assume that $T^{r}$ is positive definite on $M$, for some even integer $r\in\{0,1,\cdots,n-1\}$.
Then the eigenvalues $\lambda_i$ of the closed eigenvalue problem \eqref{maith-2} satisfy
\begin{equation}\label{2maith-3} \lambda_{1}\leq 1;
\end{equation}
\begin{equation}\label{2maith-4}
\sum_{i=1}^n\sqrt{\lambda_{i}}\leq
\sqrt{\frac{n(n-r)}{{\rm vol}(M)} \int\limits_M
S_r\,e^{-\frac{|x|^2}{2}}dv},
\end{equation} where ${\rm vol}(M)=\int_M\,e^{-\frac{|x|^2}{2}}dv$.
In particular, the equality in \eqref{2maith-4} holds if and only
if
$\lambda_{1}=\lambda_{2}=\cdots=\lambda_{N}$.

\end{thm}

Using the fact
$$\lambda_{1}\leq\lambda_{2}\leq\cdots\leq\lambda_{n},$$
we have
$$\sum\limits_{i=1}^n\sqrt{\lambda_{i}}
\geq n\sqrt{\lambda_{1}},$$ and
$$\sum\limits_{i=1}^n\sqrt{\lambda_{i}}
\geq \sqrt{\lambda_{n}}.$$ Therefore, we obtain the
following upper bound of the first eigenvalue
$\lambda_{1}$ from \eqref{2maith-4}:

\begin{corr}\label{cor1} Under the assumption of Theorem \ref{mainth2},
we have
\begin{equation}\label{cor-1}
\lambda_{1}\leq \frac{n-r}{n\,{\rm vol}(M)}
\int\limits_M S_r\,e^{-\frac{|x|^2}{2}}dv;
\end{equation}
\begin{equation}\label{cor-2}
\lambda_{n}\leq \frac{n(n-r)}{{\rm vol}(M)}
\int\limits_M S_r\,e^{-\frac{|x|^2}{2}}dv.
\end{equation} The equality in \eqref{cor-1} holds if and only
if
$\lambda_{1}=\lambda_{2}=\cdots=\lambda_{N}$.
\end{corr}

\section{A general inequality for eigenvalues of the operator $\mathcal{L}^{(f,T)}$}

In this section, we prove the following general inequalities for
eigenvalues of the elliptic operator $\mathcal{L}^{(f,T)}$ defined by
\eqref{1Introduction-Section5} in divergence form with a weight on
compact Riemannian manifolds.

\begin{thm}\label{2Sectionthm1}
Let $\lambda_{i}$ be the $i$-th eigenvalue of the problem
\eqref{1Introduction-Section13} and let $u_i$ be the normalized
eigenfunction corresponding to $\lambda_{i}$ such that
$\{u_i\}_1^\infty$ becomes an orthonormal basis of $L^2(M)$ under
the weighted measure $d\mu=e^{-f}dv$, that is
\begin{equation}\label{2Section-Ineq1}
\left\{\begin{array}{l} \mathcal{L}^{(f,T)}(u_i)=-\lambda_{i} u_i, \ {\rm in}\ M;\\
u_i|_{\partial M}=0;\\
 \int_{M}u_iu_j\,d\mu=\delta_{ij}.
\end{array}\right.
\end{equation}
Then for any function $h\in  C^2(M)$ and any positive integer $k$, we have
\begin{equation}\label{2Section-Ineq2}\aligned
&\sum_{i=1}^k(\lambda_{k+1}-\lambda_{i})^2\int\limits_M
u_i^2\langle\nabla h,T\nabla h\rangle\,d\mu\\
\leq &
\sum_{i=1}^k(\lambda_{k+1}-\lambda_{i})
\int\limits_M\left(u_i\mathcal{L}^{(f,T)}(h)+2\langle
\nabla u_i,T\nabla h\rangle\right)^2\,d\mu;
\endaligned
\end{equation}
and
\begin{equation}\label{2Section-Ineq3}\aligned
\sum_{i=1}^k(\lambda_{k+1}-\lambda_i)^2\int\limits_M u_i^2|\nabla
h|^2\,d\mu
\leq\delta \sum_{i=1}^k(\lambda_{k+1}-\lambda_i)^2\int\limits_M
u_i^2\langle\nabla h,T\nabla
h\rangle\,&d\mu\\
+\frac{1}{\delta}\sum_{i=1}^k(\lambda_{k+1}-\lambda_i)
\int\limits_M\Big(\frac{1}{2}u_i\Delta_f
h+\langle\nabla u_i,\nabla h\rangle\Big)^2\,&d\mu,
\endaligned
\end{equation}
where $\delta$ is a positive constant.
Furthermore, if there exists a function $h_A\in C^2(M)$ satisfying
\begin{equation}\label{2Section-Ineq4}
\int\limits_M h_A u_1u_B\,d\mu=0, \ \ \ {\rm for}\ B=2,\cdots,A,
\end{equation} then we get
\begin{equation}\label{2Section-Ineq5}
(\lambda_{A+1}-\lambda_{1})\int\limits_M u_1^2\langle \nabla
h_A,T\nabla h_A\rangle\,d\mu\leq\int\limits_M
[u_1\mathcal{L}^{(f,T)}(h_A) +2\langle \nabla u_1,T\nabla
h_A\rangle]^2\,d\mu;
\end{equation}
and
\begin{equation}\label{2Section-Ineq6}\aligned
\sqrt{\lambda_{A+1}-\lambda_{1}}\int\limits_Mu_1^2|\nabla
h_A|^2\,d\mu &\leq\delta \int\limits_M u_1^2\langle \nabla
h_A,T\nabla h_A\rangle\,d\mu\\
&+\frac{1}{\delta}\int\limits_M \Big(\langle\nabla h_A,\nabla
u_1\rangle +\frac{1}{2}u_1\Delta_fh_A\Big)^2\,d\mu.
\endaligned\end{equation}
\end{thm}

\begin{rem}
When $T$ is an identity map, the estimate \eqref{2Section-Ineq2}
becomes the estimate (1.5) in Theorem 1.1 of Xia and Xu
\cite{XiaXu2014}; When $f=0$, the estimate \eqref{2Section-Ineq3}
becomes the estimate (2.4) in Theorem 2.1 of do Carmo, Wang and Xia
\cite{docarm2010} with $V=0$ and $\rho=1$. For the recent
developments about universal inequalities for eigenvalues of the
Laplace operator on Riemannian manifolds, we refer to
\cite{Chen2011,Cheng2005,Cheng2006,ChengYang2006,ChengYang2014,
Soufi2009,Jost2010,Jost2011,Wang2007,ZhangMa2013,Wangxia2007,Yang1991}
and the references therein.

\end{rem}

\begin{rem}

When $x: M\rightarrow \mathbb{R}^{N}$ is a compact self-shrinker, choosing
$T=I$  and $f=\frac{|x|^2}{2}$ in
\eqref{2Section-Ineq2} and \eqref{2Section-Ineq5}, respectively, we
obtain
\begin{equation}\label{100Section1}\aligned
&\sum_{i=1}^k(\lambda_{k+1}-\lambda_{i})^2\int_M u_i^2|\nabla
h|^2 d\mu\\
\leq&
\sum_{i=1}^k(\lambda_{k+1}-\lambda_{i})
\int\limits_M[u_i\mathcal{L}(h)+2\langle
\nabla u_i,\nabla h\rangle]^2 d\mu;
\endaligned\end{equation}
and
\begin{equation}\label{100Section2}
(\lambda_{A+1}-\lambda_{1})\int\limits_M u_1^2\langle \nabla
h_A,\nabla h_A\rangle\,d\mu\leq\int\limits_M [u_1\mathcal{L}(h_A)
+2\langle \nabla u_1,\nabla h_A\rangle]^2\,d\mu.
\end{equation} Replacing $h$ and $h_A$ by $x_A$ in
\eqref{100Section1} and \eqref{100Section2} and summing over $A$, we
derive Theorem 1.1 and Proposition 5.1 of Cheng and Peng in
\cite{ChengPeng2013}, respectively. Here $x_A, A=1,\cdots,N,$ denote components
of the position vector $x$.

\end{rem}

\begin{proof}[Proof of the estimate \eqref{2Section-Ineq2}]
Let \begin{equation}\label{2Section-Ineq7}
\varphi_i=hu_i-\sum_{j=1}^ka_{ij}u_j,
\end{equation}
where
$$a_{ij}=\int\limits_M h u_iu_j\,d\mu=a_{ji}.$$
It is easy to see that
$$\varphi_i|_{\partial M}=0,\ \ \ \int\limits_M \varphi_iu_j\,d\mu=0,
\ \ {\rm for}\ \forall \ i,j=1,2,\cdots,k.$$ We have from the
Rayleigh-Ritz inequality
\begin{equation}\label{2Section-Ineq8}
\lambda_{k+1}\int\limits_M \varphi_i^2\,d\mu\leq-\int\limits_M
\varphi_i\mathcal{L}^{(f,T)}(\varphi_i)\,d\mu.
\end{equation}
Putting
\begin{equation}\label{2Section-Ineq9}\aligned
\mathcal{L}^{(f,T)}(\varphi_i)&
=\mathcal{L}^{(f,T)}(hu_i)+\sum_{j=1}^ka_{ij}\lambda_ju_j\\
&=-\lambda_ihu_i+u_i\mathcal{L}^{(f,T)}(h)+2\langle\nabla
u_i,T\nabla h\rangle+\sum_{j=1}^ka_{ij}\lambda_ju_j
\endaligned\end{equation}
in \eqref{2Section-Ineq8} gives
\begin{equation}\label{2Section-Ineq10}\aligned
&\lambda_{k+1}\int\limits_M
\varphi_i^2\,d\mu\\
 \leq \ &\lambda_i\int\limits_M
\varphi_i(hu_i)\,d\mu-\int\limits_M
\varphi_i\left(u_i\mathcal{L}^{(f,T)}(h)+2\langle\nabla
u_i,T\nabla h\rangle\right)\,d\mu\\
=\ &\lambda_i\int\limits_M \varphi_i^2\,d\mu-\int\limits_M
\varphi_i\left(u_i\mathcal{L}^{(f,T)}(h)+2\langle\nabla u_i,T\nabla
h\rangle\right)\,d\mu,
\endaligned\end{equation} which shows that
\begin{equation}\label{2Section-Ineq11}
(\lambda_{k+1}-\lambda_i)\int\limits_M\varphi_i^2\,d\mu\leq P_i,
\end{equation} where
\begin{equation}\label{2Section-Ineq12}\aligned
P_i=&-\int\limits_M\varphi_i\left(u_i\mathcal{L}^{(f,T)}(h)+2\langle\nabla
u_i,T\nabla h\rangle\right)\,d\mu\\
=&-\int\limits_M hu_i\left(u_i\mathcal{L}^{(f,T)}(h)+2\langle\nabla
u_i,T\nabla
h\rangle\right)\,d\mu+\sum_{j=1}^ka_{ij}b_{ij}
\endaligned
\end{equation}
with $$b_{ij}=\int\limits_M u_j\left(u_i\mathcal{L}^{(f,T)}(h)
+2\langle\nabla
u_i,T\nabla h\rangle\right)\,d\mu.$$ Noticing
$$\aligned
\lambda_ia_{ij}&=-\int\limits_M \mathcal{L}^{(f,T)}(u_i)h u_j\,d\mu\\
&=-\int\limits_M\mathcal{L}^{(f,T)}(h u_j)u_i\,d\mu\\
&=-\int\limits_M\left(h\mathcal{L}^{(f,T)} (u_j)+u_j\mathcal{L}^{(f,T)}
(h)+2\langle\nabla u_j,T\nabla h\rangle\right)u_i\,d\mu\\
&=\lambda_ja_{ij}-\int\limits_M\left(u_j\mathcal{L}^{(f,T)}
(h)+2\langle\nabla
u_j,T\nabla h\rangle\right)u_i\,d\mu\\
&=\lambda_ja_{ij}-b_{ji}.
\endaligned$$
Hence, we have
\begin{equation}\label{2Section-Ineq13}
b_{ji}=(\lambda_j-\lambda_i)a_{ij}=-b_{ij}.
\end{equation}
By virtue of the Stokes formula, we have
\begin{equation}\label{2Section-Ineq14}
-\int\limits_M hu_i\left(u_i\mathcal{L}^{(f,T)}(h)+2\langle\nabla
u_i,T\nabla h\rangle\right)\,d\mu=\int\limits_M u_i^2\langle\nabla
h,T\nabla h\rangle\,d\mu.
\end{equation}
Thus, we have from \eqref{2Section-Ineq12}
\begin{equation}\label{2Section-Ineq15}
P_i=\int\limits_M u_i^2\langle\nabla h,T\nabla
h\rangle\,d\mu+\sum_{j=1}^k(\lambda_i-\lambda_j)a^{2}_{ij}.
\end{equation}
By the Schwarz inequality and \eqref{2Section-Ineq11}, we infer
\begin{equation}\label{2Section-Ineq16}\aligned
&(\lambda_{k+1}-\lambda_i)P_i^2\\
=\ &(\lambda_{k+1}-\lambda_i)\left(\int\limits_M
\varphi_i\Big(u_i\mathcal{L}^{(f,T)}(h)+2\langle\nabla u_i,T\nabla
h\rangle-\sum_{j=1}^kb_{ij}u_j\Big)\,d\mu\right)^2\\
\leq \ &(\lambda_{k+1}-\lambda_i)
\int\limits_M\varphi_i^2\,d\mu
\cdot\left(\int\limits_M
\left(u_i\mathcal{L}^{(f,T)}(h)+2\langle\nabla
u_i,T\nabla h\rangle\right)^2\,d\mu-\sum_{j=1}^kb_{ij}^2\right)\\
\leq \ &P_i\left(\int\limits_M
\left(u_i\mathcal{L}^{(f,T)}(h)+2\langle\nabla
u_i,T\nabla h\rangle\right)^2\,d\mu-\sum_{j=1}^kb_{ij}^2\right),
\endaligned\end{equation} which gives
\begin{equation}\label{2Section-Ineq17}\aligned
(\lambda_{k+1}-\lambda_i)P_i\leq
\int\limits_M\left(u_i\mathcal{L}^{(f,T)}(h)+2\langle\nabla u_i,T\nabla
h\rangle\right)^2\,d\mu-\sum_{j=1}^kb_{ij}^2.
\endaligned\end{equation}
Multiplying \eqref{2Section-Ineq17} by $(\lambda_{k+1}-\lambda_i)$
and taking sum on $i$ from 1 to $k$, we get
\begin{equation}\label{add2Section-Ineq17}\aligned
\sum_{i=1}^k(\lambda_{k+1}&-\lambda_i)^2 P_i\leq
-\sum_{i,j=1}^k(\lambda_{k+1}-\lambda_i)b_{ij}^2\\
&+\sum_{i=1}^k(\lambda_{k+1}-\lambda_i)
\int\limits_M\left(u_i\mathcal{L}^{(f,T)}(h)+2\langle\nabla
u_i,T\nabla h\rangle\right)^2\,d\mu.
\endaligned\end{equation}
Applying the inequality \eqref{2Section-Ineq15}, $a_{ij}=a_{ji}$ and
$b_{ij}=-b_{ji}$ into \eqref{add2Section-Ineq17} concludes the proof
of the estimate \eqref{2Section-Ineq2}.
\end{proof}

\begin{proof}[Proof of the estimate \eqref{2Section-Ineq3}]
From \eqref{2Section-Ineq11} and \eqref{2Section-Ineq15}, we have
obtained
\begin{equation}\label{2Section-Ineq18}
(\lambda_{k+1}-\lambda_i)\int\limits_M
\varphi_i^2\,d\mu\leq\int\limits_M u_i^2\langle\nabla h,T\nabla
h\rangle\,d\mu+\sum_{j=1}^k(\lambda_i-\lambda_j)a^2_{ij}.
\end{equation}
Let $$c_{ij}=\int\limits_M u_j\Big(\frac{1}{2}u_i\Delta_f
h+\langle\nabla u_i,\nabla h\rangle\Big)\,d\mu.$$ We have from
the Stokes formula
$$c_{ij}+c_{ji}=\int\limits_M\Big(u_iu_j\Delta_f h+\langle\nabla
(u_iu_j),\nabla h\rangle\Big)\,d\mu=0$$ and
\begin{equation}\label{2Section-Ineq19}\aligned &-2\int\limits_M
\varphi_i\Big(\frac{1}{2}u_i\Delta_f h+\langle\nabla u_i,\nabla
h\rangle\Big)\,d\mu\\
=&-2\int\limits_Mhu_i\Big(\frac{1}{2}u_i\Delta_f h+\langle\nabla
u_i,\nabla
h\rangle\Big)\,d\mu+2\sum_{j=1}^kc_{ij}a_{ij}\\
=&\int\limits_Mu_i^2|\nabla h|^2\,d\mu+2\sum_{j=1}^kc_{ij}a_{ij}.
\endaligned\end{equation}
Multiplying \eqref{2Section-Ineq19} by $(\lambda_{k+1}-\lambda_i)^2$
and using the Schwarz inequality and \eqref{2Section-Ineq18}, we
obtain $$\aligned  &(\lambda_{k+1}-\lambda_i)^2\Big(\int\limits_M
u_i^2|\nabla h|^2\,d\mu+2\sum_{j=1}^kc_{ij}a_{ij}\Big)\\
\leq&-2(\lambda_{k+1}-\lambda_i)^2
\int\limits_M\varphi_i\Big(\frac{1}{2}u_i\Delta_f
h+\langle\nabla u_i,\nabla h\rangle-\sum_{j=1}^kc_{ij}u_j\Big)\,d\mu\\
\leq \ &\delta(\lambda_{k+1}-\lambda_i)^3
\int\limits_M\varphi_i^2\,d\mu\\
+&\frac{\lambda_{k+1}-\lambda_i}{\delta}\Big(\int\limits_M\Big(\frac{1}{2}u_i\Delta_f
h+\langle\nabla u_i,\nabla
h\rangle\Big)^2\,d\mu-\sum_{j=1}^kc_{ij}^2\Big)\\
\leq \ &\delta(\lambda_{k+1}-\lambda_i)^2\Big(\int\limits_M
u_i^2\langle\nabla h,T\nabla
h\rangle\,d\mu+\sum_{j=1}^k(\lambda_i-\lambda_j)a^{2}_{ij}\Big)\\
+&\frac{\lambda_{k+1}-\lambda_i}
{\delta}\Big(\int\limits_M\Big(\frac{1}{2}u_i\Delta_f
h+\langle\nabla u_i,\nabla
h\rangle\Big)^2\,d\mu-\sum_{j=1}^kc_{ij}^2\Big),
\endaligned$$ where $\delta$ is any positive constant. Summing over
$i$ and noticing $a_{ij}=a_{ji}$, $c_{ij}=-c_{ji}$, we conclude
$$\aligned \sum_{i=1}^k(\lambda_{k+1}-\lambda_i)^2\int\limits_M
u_i^2&|\nabla h|^2\,d\mu
\leq\delta\sum_{i=1}^k(\lambda_{k+1}-\lambda_i)^2
\int\limits_M
u_i^2\langle\nabla h,T\nabla
h\rangle\,d\mu\\
&+\frac{1}{\delta}\sum_{i=1}^k(\lambda_{k+1}-\lambda_i)\int\limits_M\Big(\frac{1}{2}u_i\Delta_f
h+\langle\nabla u_i,\nabla h\rangle\Big)^2\,d\mu.
\endaligned$$
We complete the proof of the estimate \eqref{2Section-Ineq3}.
\end{proof}

\begin{proof}[Proof of the estimate \eqref{2Section-Ineq5}]
We let $\varphi_A=h_A u_1-u_1\int_M h_A u_1^2\,d\mu$. Then
\begin{equation}\label{2Section-Ineq20} \int\limits_M  \varphi_A u_1\,d\mu=0.
\end{equation} It has been shown  from \eqref{2Section-Ineq4} that
\begin{equation}\label{2Section-Ineq21}
\int\limits_M \varphi_A u_B\,d\mu=0, \ \ \ {\rm for}\ B=2,\cdots,A.
\end{equation} Hence, we have from the Rayleigh-Ritz inequality
\begin{equation}\label{2Section-Ineq22}
\lambda_{A+1}\int\limits_M\varphi_A
^2\,d\mu\leq-\int\limits_M\varphi_A
\mathcal{L}^{(f,T)}(\varphi_A)\,d\mu.
\end{equation}
According to the Stokes formula, a direct calculation yields
\begin{equation}\label{2Section-Ineq23}\aligned
&-\int\limits_M\varphi_A \mathcal{L}^{(f,T)}(\varphi_A)\,d\mu\\
=&-\int\limits_M\varphi_A\,
\mathcal{L}^{(f,T)}(h_A u_1)\,d\mu\\
=&-\int\limits_M \varphi_A\left(-\lambda_1h_A
u_1+u_1\mathcal{L}^{(f,T)}(h_A) +2\langle \nabla u_1,T\nabla
h_A\rangle\right)\,d\mu\\
=&\lambda_1\int\limits_M \varphi_A^2\,d\mu-\int\limits_M
\varphi_A\left(u_1\mathcal{L}^{(f,T)}(h_A) +2\langle \nabla u_1,T\nabla
h_A\rangle\right)\,d\mu.
\endaligned
\end{equation}
Putting \eqref{2Section-Ineq23} into the inequality
\eqref{2Section-Ineq22} gives
\begin{equation}\label{2Section-Ineq24}\aligned
&(\lambda_{A+1}-\lambda_{1})\int\limits_M \varphi_A^2\,d\mu\\
\leq&
-\int\limits_M \varphi_A\left(u_1\mathcal{L}^{(f,T)}(h_A) +2\langle
\nabla
u_1,T\nabla h_A\rangle\right)\,d\mu\\
=&-\int\limits_M h_Au_1\left(u_1\mathcal{L}^{(f,T)}(h_A) +2\langle \nabla
u_1,T\nabla h_A\rangle\right)\,d\mu\\
=&\int\limits_M u_1^2\langle \nabla h_A,T\nabla h_A\rangle)\,d\mu.
\endaligned\end{equation}

We define \begin{equation}\label{addlemma2proof25}
\omega_A:=-\int\limits_M \varphi_A[u_1\mathcal{L}^{(f,T)}(h_A)
+2\langle \nabla u_1,T\nabla h_A\rangle]\,d\mu=\int\limits_M
u_1^2\langle \nabla h_A,T\nabla h_A\rangle)\,d\mu.
\end{equation} Then \eqref{2Section-Ineq24} gives
\begin{equation}\label{2Section-Ineq26}
(\lambda_{A+1}-\lambda_{1})\int\limits_M \varphi_A^2\,d\mu\leq
\omega_A.
\end{equation}
From the Schwarz inequality and \eqref{2Section-Ineq26}, we obtain
\begin{equation}\label{2Section-Ineq27}\aligned
&(\lambda_{A+1}-\lambda_{1})\omega_A^2
\\
=\ &(\lambda_{A+1}-\lambda_{1})\left(\int\limits_M
\varphi_A\left(u_1\mathcal{L}^{(f,T)}(h_A)
+2\langle \nabla u_1,T\nabla h_A\rangle\right)\,d\mu\right)^2\\
\leq \ &(\lambda_{A+1}-\lambda_{1})\int\limits_M
\varphi_A^2\,d\mu\cdot\int\limits_M \left(u_1\mathcal{L}^{(f,T)}(h_A)
+2\langle \nabla u_1,T\nabla h_A\rangle\right)^2\,d\mu\\
\leq\ &\omega_A\int\limits_M \left(u_1\mathcal{L}^{(f,T)}(h_A) +2\langle
\nabla u_1,T\nabla h_A\rangle\right)^2\,d\mu,
\endaligned\end{equation} which gives
\begin{equation}\label{2Section-Ineq28}
(\lambda_{A+1}-\lambda_{1})\omega_A\leq\int\limits_M
\left(u_1\mathcal{L}^{(f,T)}(h_A) +2\langle \nabla u_1,T\nabla
h_A\rangle\right)^2\,d\mu.
\end{equation} Combining \eqref{2Section-Ineq28} with
\eqref{addlemma2proof25} yields the estimate \eqref{2Section-Ineq5}.
\end{proof}

\begin{proof}[Proof of the estimate \eqref{2Section-Ineq6}]
On the other hand, from the the Stokes formula again, one gets
\begin{equation}\label{2Section-Ineq29}\aligned
&-\int\limits_M \varphi_A\Big(\langle\nabla h_A,\nabla u_1\rangle
+\frac{1}{2}u_1\Delta_fh_A\Big)\,d\mu\\
=&-\int\limits_M h_A
u_1\Big(\langle\nabla h_A,\nabla u_1\rangle
+\frac{1}{2}u_1\Delta_fh_A\Big)\,d\mu\\
= & \ \frac{1}{2}\int\limits_Mu_1^2|\nabla h_A|^2\,d\mu.
\endaligned\end{equation}
Therefore, for any positive constant $\delta$, we derive from
\eqref{2Section-Ineq29}
\begin{equation}\label{2Section-Ineq30} \aligned
&\sqrt{\lambda_{A+1}-\lambda_{1}}\int\limits_Mu_1^2|\nabla
h_A|^2\,d\mu\\
=&-2\sqrt{\lambda_{A+1}-\lambda_{1}}\int\limits_M
\varphi_A\Big(\langle\nabla h_A,\nabla u_1\rangle
+\frac{1}{2}u_1\Delta_fh_A\Big)\,d\mu\\
\leq \ &\delta(\lambda_{A+1}-\lambda_{1})\int\limits_M
\varphi_A^2\,d\mu+\frac{1}{\delta}\int\limits_M \Big(\langle\nabla
h_A,\nabla u_1\rangle
+\frac{1}{2}u_1\Delta_fh_A\Big)^2\,d\mu\\
\leq \ &\delta \int\limits_M u_1^2\langle \nabla h_A,T\nabla
h_A\rangle)\,d\mu+\frac{1}{\delta}\int\limits_M \Big(\langle\nabla
h_A,\nabla u_1\rangle +\frac{1}{2}u_1\Delta_fh_A\Big)^2\,d\mu,
\endaligned
\end{equation} where in the last inequality we used \eqref{2Section-Ineq24}.
Hence, the desired estimate \eqref{2Section-Ineq6} is obtained.
\end{proof}

\section{Proof of theorems}

\begin{proof}[Proof of Theorem \ref{mainth1}]
When $x: M\rightarrow \mathbb{R}^{N}$ is a compact self-shrinker,
substituting $T$ and $f$ by $T^r$ and $\frac{|x|^2}{2}$ in
\eqref{2Section-Ineq3} and \eqref{2Section-Ineq6}, respectively, we
obtain for any positive constant $\delta$,
\begin{equation}\label{4Proof1}\aligned
&\sum_{i=1}^k(\lambda_{k+1}-\lambda_i)^2\int\limits_M u_i^2|\nabla
h|^2\,e^{-\frac{|x|^2}{2}}dv\\
\leq\delta &\sum_{i=1}^k(\lambda_{k+1}-\lambda_i)^2\int\limits_M
u_i^2\langle\nabla h,T^r\nabla
h\rangle\,e^{-\frac{|x|^2}{2}}dv\\
+\frac{1}{\delta}&\sum_{i=1}^k(\lambda_{k+1}-\lambda_i)\int\limits_M\Big(\frac{1}{2}u_i\mathcal{L}
(h)+\langle\nabla u_i,\nabla
h\rangle\Big)^2\,e^{-\frac{|x|^2}{2}}dv,
\endaligned\end{equation}
and
\begin{equation}\label{4Proof2}\aligned
\sqrt{\lambda_{A+1}-\lambda_{1}}\int\limits_Mu_1^2&|\nabla
h_A|^2\,e^{-\frac{|x|^2}{2}}dv\leq\delta \int\limits_M u_1^2\langle
\nabla
h_A,T^r\nabla h_A\rangle)\,e^{-\frac{|x|^2}{2}}dv\\
&+\frac{1}{\delta}\int\limits_M \Big(\langle\nabla h_A,\nabla
u_1\rangle
+\frac{1}{2}u_1\mathcal{L}(h_A)\Big)^2\,e^{-\frac{|x|^2}{2}}dv.
\endaligned\end{equation}

Let $E_1,\cdots,E_N$ be a canonical orthonormal basis of
$\mathbb{R}^N$. Then $x_A=\langle E_A,x\rangle$ and
\begin{equation}\label{4Proof3}\nabla x_A =\langle
E_A,e_i\rangle e_i=E_A^{\top},\end{equation}
 where $\top$ denotes the tangent projection to $M$. Therefore,
\begin{equation}\label{4Proof4} |\nabla
x_A|^2=|E_A^{\top}|^2\leq|E_A|^2=1, \ \ \ \ \forall \ A;
\end{equation}
\begin{equation}\label{4Proof5}
\sum_{A}|\nabla x_A|^2=n;
\end{equation}
\begin{equation}\label{4Proof6}
\sum_{A}\langle\nabla x_A,\nabla u_i\rangle^2=|\nabla u_i|^2;
\end{equation}
\begin{equation}\label{4Proof7}\aligned\sum_{A}\langle
\nabla x_A,T^r\nabla x_A\rangle&=\sum_{A}T^r_{ij}\langle
E_A,e_i\rangle\langle E_A,e_j\rangle\\
&=T^r_{ij}\langle e_i,e_j\rangle\\
&={\rm trace}(T^r)\\
&=(n-r)S_r;\endaligned\end{equation}
\begin{equation}\label{4Proof8}\aligned
\mathcal{L}(x_A)=&\Delta(x_A)-\langle x,\nabla x_A \rangle\\
=&\langle n\mathbf{H},E_A\rangle-\langle x,E_A^{\top}\rangle\\
=&-\langle x^{\perp},E_A\rangle-\langle x^{\top},E_A\rangle\\
=&-\langle x,E_A\rangle\\
=&-x_A;
\endaligned\end{equation}
\begin{equation}\label{4Proof9}\aligned
&\sum_{A}\int\limits_Mu_i\mathcal{L}(x_A) \langle\nabla u_i,\nabla
x_A\rangle\,e^{-\frac{|x|^2}{2}}dv\\
=&-\sum_{A}\int\limits_Mu_ix_A
\langle\nabla u_i,\nabla x_A\rangle\,e^{-\frac{|x|^2}{2}}dv\\
=\ &\frac{1}{4}\sum_{A}\int\limits_Mu_i^2
\mathcal{L}(x_A^2)\,e^{-\frac{|x|^2}{2}}dv\\
=\ &\frac{1}{2}\int\limits_Mu_i^2(n-|x|^2)\,e^{-\frac{|x|^2}{2}}dv.
\endaligned\end{equation} Here in \eqref{4Proof7}, we used Lemma 3.3
in \cite{Cao2007} which is still valid for self-shrinkers. Taking
$h=x_A$ in \eqref{4Proof1} and summing over $A$, we get
\begin{equation}\label{4Proof10}\aligned
&n\sum_{i=1}^k(\lambda_{k+1}-\lambda_i)^2\\
\leq \ &(n-r)\delta\sum_{i=1}^k(\lambda_{k+1}-\lambda_i)^2
\int\limits_M
u_i^2S_r\,e^{-\frac{|x|^2}{2}}dv\\
+ \ &  \frac{1}{\delta}\sum_{i=1}^k
(\lambda_{k+1}-\lambda_i)\int\limits_M
\Big(\frac{1}{4}u_i|x|^2
+|\nabla
u_i|^2+\frac{1}{2}u_i^2(n-|x|^2)\Big)\,
e^{-\frac{|x|^2}{2}}dv\\
\leq \ &(n-r)\max_M(S_r)\delta
\sum_{i=1}^k(\lambda_{k+1}-\lambda_i)^2\\
 + \ & \frac{1}{\delta}\sum_{i=1}^k
(\lambda_{k+1}-\lambda_i)\Big(\frac{\lambda_i}{\xi}+
\frac{2n-\min\limits_{M}|x|^2}{4}\Big).
\endaligned\end{equation}
Minimizing the right hand side of \eqref{4Proof10} by taking
$$\delta=\sqrt{\frac{\sum\limits_{i=1}^k
(\lambda_{k+1}-\lambda_i)\Big(\frac{\lambda_i}{\xi}+\frac{n}{2}
-\frac{1}{4}\min\limits_{M}|x|^2\Big)}
{(n-r)\sum\limits_{i=1}^k
(\lambda_{k+1}-\lambda_i)^2\max\limits_M(S_r)}}$$
gives
\begin{equation}\label{4Proof11}\aligned
\sum_{i=1}^k(\lambda_{k+1}-\lambda_i)^2
\leq\frac{4(n-r)}{n^2}\max_M(S_r)
\sum_{i=1}^k(\lambda_{k+1}-\lambda_i)\Big(\frac{\lambda_i}{\xi}+
\frac{2n-\min\limits_{M}|x|^2}{4}\Big).
\endaligned\end{equation}

On the other hand, according to the orthogonalization of Gram and
Schmidt, we get that there exists an  orthogonal $N\times N$-matrix $O=(O_A^B)$
such that
\begin{equation}\label{4Proof12}
\sum_{C}\int\limits_M O_A^C x_Cu_1u_B
=\sum_{C}O_A^C\int\limits_M x_C u_1u_B=0, \ \ \ {\rm for}\
B=2,\ldots,A.
\end{equation} Therefore, taking $h_A=\sum_{C}O_A^C
x_C$ in \eqref{4Proof2}, and summing over $A$, we obtain
\begin{equation}\label{4Proof13}\aligned
&\sum_{A}\sqrt{\lambda_{A+1}-\lambda_{1}}
\int\limits_Mu_1^2|\nabla
h_A|^2\,e^{-\frac{|x|^2}{2}}dv\\
\leq \ &(n-r)\max_M(S_r)\delta +\frac{1}{\delta}\Big(\frac{\lambda_1}{\xi}+ \frac{2n-\min\limits_{M}|x|^2}{4}\Big).
\endaligned\end{equation} Using \eqref{4Proof4}, we infer
\begin{equation}\label{4Proof14}
\aligned &\sum_{A=1}^N\sqrt{\lambda_{A+1}-\lambda_{1}}|\nabla
   h_{A}|^2\\
   \geq&\sum_{i=1}^n\sqrt{\lambda_{i+1}-\lambda_{1}}|\nabla
   h_{i}|^2+\sqrt{\lambda_{n+1}-\lambda_{1}}
   \sum\limits_{\alpha=n+1}^N|\nabla h_{\alpha}|^2\\
   =&\sum_{i=1}^n\sqrt{\lambda_{i+1}-\lambda_{1}}|\nabla
   h_{i}|^2+\sqrt{\lambda_{n+1}-\lambda_{1}}
   \left(n-\sum\limits_{j=1}^n|\nabla
   h_{j}|^2\right)\\
   =&\sum_{i=1}^n\sqrt{\lambda_{i+1}-\lambda_{1}}|\nabla
   h_{i}|^2+\sqrt{\lambda_{n+1}-\lambda_{1}}
   \sum\limits_{j=1}^n(1-|\nabla
   h_{j}|^2)\\
   \geq &\sum_{i=1}^n\sqrt{\lambda_{i+1}-\lambda_{1}}|\nabla
   h_{i}|^2+\sum\limits_{j=1}^n
   \sqrt{\lambda_{j+1}-\lambda_{1}}(1-|\nabla
   h_{j}|^2)\\
   =&\sum\limits_{i=1}^n\sqrt{\lambda_{i+1}-\lambda_{1}}.
\endaligned
\end{equation}
Applying \eqref{4Proof14} into \eqref{4Proof13} yields
\begin{equation}\label{4Proof15}
\sum\limits_{i=1}^n\sqrt{\lambda_{i+1}-\lambda_{1}}
\leq(n-r)\max_M(S_r)\delta
+\frac{1}{\delta}\Big(\frac{\lambda_1}{\xi}+
\frac{2n-\min\limits_{M}|x|^2}{4}\Big).
\end{equation} Minimizing the right hand side of \eqref{4Proof15} by taking
$$\delta=\sqrt{\frac{\frac{\lambda_1}{\xi}+
\frac{2n-\min\limits_{M}|x|^2}{4}}{(n-r)\max\limits_M(S_r)}}$$ gives the estimate
\eqref{maith-3}. We complete the proof of Theorem \ref{mainth1}.
\end{proof}

\begin{proof}[Proof of Theorem \ref{mainth2}] The estimate
\eqref{2maith-3} follows from \eqref{4Proof8} directly.
We let
$$\varphi_A=h_A u_0-u_0\int_M h_A u_0^2\,e^{-\frac{|x|^2}{2}}dv,$$
where  $u_0$ is the eigenfunction corresponding to $\lambda_0=0$ satisfying
$$\int\limits_M u_0^2\,e^{-\frac{|x|^2}{2}}dv=u_0^{2}\,{\rm vol}(M)=1.$$
Following the proof of the estimate \eqref{2Section-Ineq6},
we derive the following result by replacing $u_1$ with $u_0$ in
\eqref{2Section-Ineq6} of Theorem \ref{2Sectionthm1}:

\begin{thm}\label{4thm1}
Let $\lambda_{i}$ be the $i$-th eigenvalue of the closed eigenvalue problem \eqref{maith-2} and $u_i$ the normalized eigenfunction
corresponding to $\lambda_{i}$ such that $\{u_i\}_0^\infty$ becomes
an orthonormal basis of $L^2(M)$ under the weighted measure
$e^{-\frac{|x|^2}{2}}dv$, that is
\begin{equation}\label{4Ineq16}
\left\{\begin{array}{l} \mathcal{L}_r(u_i)=-\lambda_{i} u_i, \ {\rm in}\ M;\\
\int\limits_Mu_iu_j\,e^{-\frac{|x|^2}{2}}dv=\delta_{ij}.
\end{array}\right.
\end{equation}
If there exists a function $h_A\in C^2(M)$ satisfying
\begin{equation}\label{4Ineq17}
\int\limits_M h_A u_0u_B\,e^{-\frac{|x|^2}{2}}dv=0, \ \ \ {\rm for}\
B=1,\cdots,A-1,
\end{equation} then we have,
for any positive constant $\delta$,
\begin{equation}\label{4Ineq19}\aligned
\sqrt{\lambda_{A}}\int\limits_M|\nabla
h_A|^2\,e^{-\frac{|x|^2}{2}}dv\leq&\delta \int\limits_M \langle
\nabla
h_A,T^r\nabla h_A\rangle)\,e^{-\frac{|x|^2}{2}}dv\\
&+\frac{1}{4\delta}\int\limits_M
(\mathcal{L}(h_A))^2\,e^{-\frac{|x|^2}{2}}dv.
\endaligned\end{equation}
\end{thm}

According to the orthogonalization of Gram and Schmidt, we get that
there exists an orthogonal matrix $O=(O_A^B)$ such that
\begin{equation}\label{4Ineq20}
\sum_{C}\int\limits_M O_A^C x_Cu_0u_B\,e^{-\frac{|x|^2}{2}}dv
=\sum_{\gamma}O_A^C\int\limits_M x_C
u_0u_B\,e^{-\frac{|x|^2}{2}}dv=0,
\end{equation}
 where $ B=1,\ldots,A-1$. Taking $h_A=\sum_{C}O_A^C x_C$ in \eqref{4Ineq19}, and
summing over $A$ from 1 to $N$, we obtain
\begin{equation}\label{4Ineq29}\aligned
\sum_{A}\sqrt{\lambda_{A}}\int\limits_M|\nabla
x_A|^2\,e^{-\frac{|x|^2}{2}}dv\leq&\delta\sum_{A} \int\limits_M
\langle \nabla
x_A,T^r\nabla x_A\rangle)\,e^{-\frac{|x|^2}{2}}dv\\
&+\frac{1}{4\delta}\sum_{A}\int\limits_M
(\mathcal{L}(x_A))^2\,e^{-\frac{|x|^2}{2}}dv.
\endaligned\end{equation}
Using \eqref{4Proof8}, we have
$$\aligned
\sum_{A}\int\limits_M
(\mathcal{L}(x_A))^2\,e^{-\frac{|x|^2}{2}}dv=&-\sum_{A}\int\limits_M
x_A\mathcal{L}(x_A)\,e^{-\frac{|x|^2}{2}}dv\\
=&\sum_{A}\int\limits_M |\nabla x_A
|^2\,e^{-\frac{|x|^2}{2}}dv=n\,{\rm vol}(M).
\endaligned$$
From the similar argument as in \eqref{4Proof14}, we infer
\begin{equation}\label{4Ineq30}\sum_{A}\sqrt{\lambda_{A}}|\nabla
x_A|^2\geq\sum_{i=1}^n\sqrt{\lambda_{i}}.
\end{equation}
Therefore, we derive from \eqref{4Ineq29}
\begin{equation}\label{4Ineq31}
\sum_{i=1}^n\sqrt{\lambda_{i}}\,{\rm vol}(M)\leq\delta(n-r) \int\limits_M S_r\,e^{-\frac{|x|^2}{2}}dv+\frac{n\,{\rm vol}(M)}{4\delta}.
\end{equation}
Minimizing the right hand side of \eqref{4Ineq31}, we derive the
desired \eqref{2maith-4}.
We complete the proof
of \eqref{2maith-4}.

If the equality in \eqref{2maith-4} occurs, then inequalities
\eqref{2Section-Ineq22}, \eqref{2Section-Ineq30} and \eqref{4Ineq30}
become equalities. Hence, we have
\begin{equation}\label{4Ineq32}
\lambda_{1}=\lambda_{2}=\cdots=\lambda_{N};
\end{equation}
\begin{equation}\label{4Ineq33}
\mathcal{L}(\varphi_A)=-2\delta\sqrt{\lambda_{1}}\,\varphi_A
\end{equation} with $2\delta\sqrt{\lambda_{1}}=1$.
We remark that the relationship \eqref{4Ineq33} is equivalent to
\eqref{4Proof8} for compact self-shrinkers. In fact, applying
$$\varphi_A=h_A u_0-u_0\int\limits_M h_A
u_0^2\,e^{-\frac{|x|^2}{2}}dv$$ into
$$\mathcal{L}(\varphi_A)=-2\delta\sqrt{\lambda_{1}}\,\varphi_A,$$
we infer
$$\Big(2\delta\sqrt{\lambda_{1}}-1\Big)h_A=
2\delta\sqrt{\lambda_{1}}u_{0}^{2}\int\limits_M h_A
\,e^{-\frac{|x|^2}{2}}dv.$$ It follows from \eqref{4Proof8} that
$$\int\limits_M h_A
\,e^{-\frac{|x|^2}{2}}dv=0.$$ Thus, we obtain
$2\delta\sqrt{\lambda_{1}^{\mathcal{L}_r}}=1$.
\end{proof}



\bibliographystyle{Plain}

\end{document}